\documentclass[reqno]{amsart}
\usepackage[T1]{fontenc}
\usepackage{dsfont}
\usepackage{mathrsfs}
\usepackage[colorlinks]{hyperref}
\usepackage{xcolor}
\usepackage[a4paper,asymmetric]{geometry}
\usepackage{mathscinet}
\usepackage{latexsym}
\usepackage{amsthm}
\usepackage{amssymb}
\usepackage{amsfonts}
\usepackage{amsmath}
\usepackage{longtable}
\usepackage{graphicx}
\usepackage{multirow}
\usepackage{multicol}

\newtheorem{theorem}{Theorem}[section]
\newtheorem{thm}[theorem]{Theorem}
\newtheorem{lemma}[theorem]{Lemma}
\newtheorem{lem}[theorem]{Lemma}
\newtheorem{remark}[theorem]{Remark}
\newtheorem{proposition}[theorem]{Proposition}
\newtheorem{prop}[theorem]{Proposition}
\newtheorem{corollary}[theorem]{Corollary}

\theoremstyle{definition}
\newtheorem{definition}[theorem]{Definition}
\newtheorem{defn}[theorem]{Definition}

 \newtheorem{nott}{Notation}
\theoremstyle{remark}
\numberwithin{equation}{section}
 \DeclareMathOperator{\RE}{Re}
 \DeclareMathOperator{\IM}{Im}
 
 \DeclareMathAlphabet{\mathpzc}{OT1}{pzc}{m}{it}
 \DeclareMathAlphabet{\mathsfsl}{OT1}{cmss}{m}{sl}

  \newcommand{\FH}{\mathfrak{H}}

\newcommand{\dif}{\mathrm{d}}

\newcommand{\abs}[1]{\left\vert#1\right\vert}
\newcommand{\set}[1]{\left\{#1\right\}}

\newcommand{\norm}[1]{\left\Vert#1\right\Vert}

\newcommand{\E}{\mathbb{E}}

\newcommand{\R}{\mathbb{R}}

 \newcommand{\tensor}[1]{\mathsf{#1}}

 \newcommand{\Rnum}{\mathbb{R}}
 \newcommand{\Cnum}{\mathbb{C}}

 \newcommand{\innp}[1]{\langle {#1}\rangle}
\newcommand{\Be}{\begin{equation}}
\newcommand{\Ee}{\end{equation}}
\newcommand{\Bs}{\begin{split}}
\newcommand{\Es}{\end{split}}
\newcommand{\Bes}{\begin{equation*}}
\newcommand{\Ees}{\end{equation*}}
\newcommand{\BT}{\begin{thm}}
\newcommand{\ET}{\end{thm}}
\newcommand{\Bp}{\begin{proof}}
\newcommand{\Ep}{\end{proof}}
\newcommand{\BL}{\begin{lem}}
\newcommand{\EL}{\end{lem}}
\newcommand{\BP}{\begin{proposition}}
\newcommand{\EP}{\end{proposition}}
\newcommand{\BC}{\begin{corollary}}
\newcommand{\EC}{\end{corollary}}
\newcommand{\BR}{\begin{remark}}
\newcommand{\ER}{\end{remark}}
\newcommand{\BD}{\begin{defn}}
\newcommand{\ED}{\end{defn}}
\newcommand{\BI}{\begin{itemize}}
\newcommand{\EI}{\end{itemize}}

\newcommand{\mi}{{\rm i}}

\allowdisplaybreaks

\begin{document}
\title[Complex Fractional Ornstein-Uhlenbeck Processes]{Parameter Estimation of Complex Fractional Ornstein-Uhlenbeck Processes with Fractional Noise }
\author[Y. Chen]{Yong Chen}
\address{School of Mathematics, Hunan University of Science and Technology, Xiangtan, 411201, Hunan, China}
\email{chenyong77@gmail.com; zhishi@pku.org.cn}
\author[Y. Hu]{Yaozhong Hu}
\address{Department of Mathematics, the University of Kansas, Lawrence, 66045, Kansas,USA}
\email{yhu@ku.edu}
\author[Z. Wang]{Zhi Wang}
\address{School of Sciences, Ningbo University of Technology,
 Ningbo 315211, Zhejiang, China}
\email{wangzhi1006@hotmail.com}
\begin{abstract}
We obtain  strong consistency and  asymptotic normality of a least squares estimator of the  drift coefficient for complex-valued Ornstein-Uhlenbeck processes  disturbed  by  fractional noise, extending the result of Y. Hu and D. Nualart, [Statist. Probab. Lett., 80 (2010), 1030-1038] to a special 2-dimensions.   The strategy is to exploit the Garsia-Rodemich-Rumsey inequality and complex fourth moment theorems. The main ingredients of this paper are the sample path regularity of a multiple Wiener-It\^{o} integral and two  equivalent conditions of complex fourth moment theorems in terms of the contractions of integral kernels and complex Malliavin derivatives.

\end{abstract}
\maketitle
\section{Introduction}
To model the Chandler wobble, or variation of latitude conerning with  the rotation of the earth, M. Arat\'{o}, A.N. Kolmogorov and Ya.G. Sinai \cite{arta3} (see also \cite{arat})
proposed the following   stochastic  linear  equation
\begin{equation}\label{cp}
  \dif Z_t=-\gamma Z_t\dif t+ \sqrt{a}\dif \zeta_t\,,\quad t\ge 0\,,
\end{equation}%
where $Z_t=X_1(t)+\mi X_2(t)$ is a complex-valued process, $\gamma=\lambda-\mi \omega,\,\lambda>0$, $a>0$ and $\zeta_t
$ is a complex Brownian motion. 
It is also suggested in \cite{arat} that the Brownian motion
in \eqref{cp} may be replaced by other processes.
In this paper we consider the statistical estimator of $\gamma$ when the complex Brownian motion $\zeta$ in  \eqref{cp} is replaced by a complex fractional Brownian motion
$\zeta_t=\frac{B^1_t+i B^2_t}{\sqrt 2} $, where
$(B^1_t, B^2_t)$ is a two dimensional fractional Brownian  motion
 with $H\in [\frac12,\frac34)$.
[We shall fix the Hurst parameter and then omit the explicit dependence of the process on the Hurst parameter.]
From now on we assume that $\zeta$ is a complex fractional Brownian motion of Hurst parameter $H\in (1/2, 3/4)$.

%
To compare with the work in \cite{huNua}, we
   write \eqref{cp} as
\begin{equation}\label{model}
\begin{bmatrix}
\dif X_1(t)\\ \dif X_2(t) \end{bmatrix}
=\begin{bmatrix} -\lambda & -\omega\\ \omega & -\lambda  \end{bmatrix}
\begin{bmatrix} X_1(t)\\ X_2(t) \end{bmatrix} \dif t +\sqrt{\frac{a}{2}}
\begin{bmatrix} dB^1_t\\ dB^2_t \end{bmatrix}\,.
\end{equation}
Thus \eqref{cp} can be considered as  a particular two dimensional
Langevin equation driven by fractional Brownian motions.
However, we find it is more convenient to use the complex valued equation \eqref{cp}.

Motivated by the work of  Hu and Nualart \cite{huNua}
we also consider a least squares estimator for $\gamma$.
To this end,  we intuitively rewrite \eqref{cp}  as
\begin{equation*}
   \dot{Z}_t+\gamma Z_t=\sqrt{a}\dot{\zeta}_t\,, \quad 0\le t\le T\,.
\end{equation*}
We minimize $\int_0^T\abs{ \dot{Z}_t+\gamma Z_t}^2\dif t$  to obtain a least squares estimator of $\gamma$ as follows.
\begin{equation}\label{hat gamma}
   \hat{\gamma}_T=-\frac{\int_0^T \bar{Z}_t\dif Z_t}{\int_0^T \abs{Z_t}^2\dif t}=\gamma-\sqrt{a}\frac{\int_0^T \bar{Z}_t\dif \zeta_t}{\int_0^T \abs{Z_t}^2\dif t}.
\end{equation}

The main results of the present paper are   the strong consistency and the asymptotic normality of the estimator $\hat{\gamma}_T$
which we state as follows.

\begin{thm}
\label{strong law}
  Let $H\in [\frac12,\frac34)$.
\begin{enumerate}
\item[(i)]
The least squares estimator
 $\hat{\gamma}_T$ is strongly consistent.  Namely,
$\hat{\gamma}_T$  converges to $\gamma$ almost surely
as $T\rightarrow \infty$.
\item[(ii)] $\sqrt{T}\left( \hat{\gamma}_T-\gamma\right)  $ is asymptotically normal. Namely,
\begin{equation}
   \sqrt{T}[\hat{\gamma}_T-\gamma]\stackrel{ {law}}{\to}  \mathcal{N}(0,\frac{1}{2 d^2 a} \tensor{C})\quad\hbox{
   as $T\rightarrow \infty$}\,,
\end{equation}
where
$\tensor{C}=\begin{bmatrix} \sigma^2+c & b\\ b & \sigma^2-c  \end{bmatrix}$ with
\begin{eqnarray}
\sigma^2
&=& \frac{2}{\Gamma(2-2H)^2}\int_{[0,\infty)^2} \dif x\dif y \frac{(xy)^{1-2H}}{(x+y)(x+\bar{\gamma})(y+\gamma)} \nonumber\\
&&\qquad + \frac{\Gamma^2(2H-1)}{2
     \lambda}\big(\frac{2}{\abs{\gamma}^{4H-2}}+\frac{1}{\gamma^{4H-2}}+\frac{1}{\bar{\gamma}^{4H-2}}\big)
     \label{limlim1}
\\
 c +\mi b
      &=&\frac{2}{\Gamma(2-2H)^2} \int_{[0,\infty)^2}\frac{(xy)^{1-2H}}{(y+\gamma)^2} \Big[ \frac{1}{x+y}+\frac{1}{x+\gamma} \Big] \dif x\dif y.    \label{ab c}\\
d&=&      \frac{\Gamma(2H-1)}{2\lambda}\Big(\frac{1}{\gamma^{2H-1}}+\frac{1}{\bar{\gamma}^{2H-1}}\Big)\,.
\label{d d}
  \end{eqnarray}
In the special case when  $H=\frac12$, we have
\begin{equation}
   \sqrt{T}[\hat{\gamma}_T-\gamma]\stackrel{ {law}}{\to}  \mathcal{N}(0,\frac{\lambda}{4a}\mathrm{Id}_2)\,,
\end{equation}
where $\mathrm{Id}_2$ is a $2\times 2$ identity matrix.
\end{enumerate}
\end{thm}
\begin{remark} An important new feature for the case of  fractional Ornstein-Uhlenbeck process ($H\in (1/2, 3/4)$) is that  the limiting distribution is no longer independent Gaussian as in the case of Brownian motion case ($H=1/2$). We will discuss exclusively the case $H\not=1/2$ since the case $H=1/2$ is easy.
\end{remark}

A minor difference between the
 case of  one dimensional fractional
Orstein-Uhlenbeck process considered in \cite{huNua} and our complex case is that in our least squares  estimator $\hat\gamma$ defined by
\eqref{hat gamma}, we have $\int_0^T \bar{Z}_t\dif Z_t$
in the numerator, while in \cite{huNua} it is
$\int_0^T X_t\dif X_t$.   However, this minor difference causes a
big  unpleasant  trouble. By using It\^o formula the later is expressed as $X_T^2$ plus
another manageable  term.  This is critical in the proof of
the strong consistency of the estimator since it allows us to use a famous theorem of Pickands in \cite{huNua}.  However,
we cannot no longer apply It\^o formula to $\int_0^T \bar{Z}_t\dif Z_t$ to obtain similar identity. To get around this difficulty we shall use another famous result,  the Garsia-Rodemich-Rumsey inequality \cite[Theorem 2.1]{hyz 17}.
%
%

%

To show the asymptotic normality, we may use a multi-dimensional
fourth moment theorem. However, we develop a complex version of the fourth moment theorem which is easier to use   in our case. To state the theorem we
denote $\alpha_H=H(2H-1)$ and  $\phi(s,t) =\alpha_H \abs{s-t}^{2H-2}$ and define the Hilbert space
\begin{align}
   \FH&:=L^2_{\phi}=\{f |\,f:\R_{+}\to \R,\abs{f}_{\phi}^2:=\int_0^{\infty}\int_0^{\infty}f(s)f(t)\phi(s,t)\dif s\dif t <\infty \}.\label{e.1.9}
\end{align}
Now the  theorem is stated as follows.
\begin{thm}[Fourth Moment Theorems]\label{equilent cond}
Let   $\set{F_{k}=I_{m,n}(f_k)}$ with $f_k\in \FH_{\Cnum}^{\odot m}\otimes \FH_{\Cnum}^{\odot n}$  be  a  sequence of  $(m,n)$-th complex Wiener-It\^{o} multiple integrals
(see the next section for  a discussion),   with $m$ and $n$ fixed and
$m+n\ge 2$.  Suppose that as $k\to \infty$, $E[\abs{F_k}^2]\to \sigma^2$ and $E[F_k^2]\to c+\mi b $, where $\abs{\cdot}$ is the absolute value (or modulus) of a complex number and  $c,b\in \Rnum$. Then   the following statements  are equivalent:
\begin{itemize}
\item[\textup{(i)}]The sequence $(\RE F_k,\,\IM  F_k)$ converges in law to a bivariate normal distribution with covariance matrix $\tensor{C}= \frac{1}{2}\begin{bmatrix}\sigma^2+c & b\\ b & \sigma^2-c  \end{bmatrix}$,
\item[\textup{(ii)}]$E[\abs{F_k}^4]\to c^2+b^2+2\sigma^4$.
\item[\textup{(iii)}]  $\norm{f_k\otimes_{i,j} f_k}_{\FH^{\otimes ( 2(l-i-j))}}\to 0$ and $\norm{f_k{\otimes}_{i,j} h_k}_{\FH^{\otimes ( 2(l-i-j))}}\to 0 $ for any $0<i+j\le l-1$ where $l=m+n$ and $h_k$ is the kernel of $\bar{F}_k$, i.e., $\bar{F}_k=I_{n,m}(h_k)$.
\item[\textup{(iv)}] $\norm{D F_k}^2_{\FH}$, $\norm{D \bar{F}_k}^2_{\FH}$ and $\innp{D F_k,\, {D} \bar{F}_k}_{\FH}$ converge to a constant in $L^2(\Omega)$ as $k$ tends to infinity, where $D$ is the complex Malliavin derivatives. That is to say, $\mathrm{Var}(\norm{D F_k}^2_{\FH})\to 0$, $\mathrm{Var}(\norm{D \bar{ F}_k}^2_{\FH})\to 0$ and $\mathrm{Var}(\innp{D F_k,\, {D} \bar{F}_k}_{\FH} )\to 0$ as $k$ tends to infinity.
\end{itemize}
\end{thm}
\begin{remark}
   \begin{itemize}
   \item[\textup{1)}]  If $m=n$ and $E[{F_k}^2]=0 $ or if $m\neq n$, then $\tensor{C}=\frac{\sigma^2}{2}\mathrm{Id}_2$. That is to say, the limit is a  complex Gaussian variable $\mathcal{CN}(0,\sigma^2)$ in this case. Theorem 7 of \cite{nourorti} is concerning multi-dimensional fourth moment theorems, but it requires $\tensor{C}=\frac{\sigma^2}{2}\mathrm{Id}$. Thus, our findings are partly more general.
  \item[\textup{2)}]   We give a different and simpler proof
  of the theorem.  The equivalence   $(\mathrm{i})\Leftrightarrow (\mathrm{ii})$ is shown by an indirect method in \cite{cl2} and by stein's method in \cite{camp}. In this paper, we show that $(\mathrm{i})\Rightarrow (\mathrm{ii})\Rightarrow (\mathrm{iii})\Rightarrow (\mathrm{iv})\Rightarrow (\mathrm{i})$ directly. We make use of {\rm (iii)} to show the asymptotic normality which is simpler than to use {\rm (iv)} as in previous work \cite{huNua}.
     \end{itemize}
\end{remark}

\section{Preliminaries: complex multiple Wiener-It\^o integrals}\label{subsec 2}

Denote by $(B_t,\,t\ge 0)$ a fBm of Hurst parameter $H\in (1/2, 3/4)$.  Then Gaussian isonormal process associated with $\FH$ is  given by Wiener integrals with respect to fBm for any deterministic kernel  $\in \FH$ (where $\FH$ is defined by \eqref{e.1.9}):
\begin{align}
  B(f)=\int_0^{\infty} f(s)\dif B_s,\quad \forall f\in \FH.
\end{align}

Let $\tilde{B}(\cdot)$ be an independent copy of the fractional Brownian motion $B(\cdot) $.  Following the same idea  of \cite{cl2}, we
define complex Gaussian isonormal process  and complex multiple Wiener-It\^o integrals with respect to fBm as follows.   For any $f=f_1+\mi f_2 $ with $f_1,\,f_2\in \FH$, define that
\begin{align}
     \FH_{\Cnum}&:=\{f_1+\mi f_2:\,f_1,f_2\in \FH \},\quad \langle f_1+\mi f_2\,,f_1+\mi f_2\rangle_{\FH_{\Cnum}}
     = \langle f_1 \,,f_1\rangle_{\FH}  + \langle   f_2\,, f_2\rangle_{\FH}\,,  \\
     B(f)&=B(f_1)+\mi B(f_2)\,,\quad
   \zeta(f) =\frac{1}{\sqrt2}[B(f)+\mi \tilde{B}(f)].
\end{align} Then $\zeta$ is called a complex isonormal Gaussian process over $ \FH_{\Cnum}$, which is a centered  complex Gaussian family satisfying
\begin{align*}
   E[ \zeta( {h})^2]=0,\quad E[\zeta( {g})\overline{\zeta( {h})}]=\innp{ {g}, {h}}_{\FH_{\Cnum}},\quad \forall  {g}, {h}\in \FH_{\Cnum}.
\end{align*}

From now on, without ambiguity, we still denote $\FH_{\Cnum}$ by $\FH$.

\begin{definition}[Complex multiple Wiener-It\^{o} integrals]\label{imn}
For a fixed $(p,q)$, suppose that $g\in \FH^{\odot p}\otimes \FH^{\odot q}$, we call ${I}_{p,q}(g)$ the complex multiple Wiener-It\^{o} integral of $g$ with respect to $\zeta$ (see \cite{cl2}).
And if $f\in \FH^{\otimes (p+ q)}$ then we define
\begin{equation}\label{ipq}
  I_{p,q}(f)=I_{p,q}(\tilde{f}),
\end{equation}
where $\tilde{f}$ is the symmetrization of $f$ in the sense of It\^{o} \cite{ito}:
\begin{equation}\label{ito-sense}
   \tilde{f}(t_1,\dots,t_{p+q})=\frac{1}{p!q!}\sum_{\pi}\sum_{\sigma}f(t_{\pi(1)},\dots,t_{\pi(p)},t_{\sigma(1)},\dots,t_{\sigma(q)}),
\end{equation}
where $\pi$ and $\sigma$ run over all permutations of $(1,\dots,p)$ and $(p+1,\dots, p+q)$ respectively.
\end{definition}

%
It is easy to see   that $\overline{{I}_{p,q}(f)}={I}_{q,p}(\bar f)$ and
\begin{align}
   E[I_{p,q}(f) \overline{I_{p',q'}(g)}]&=\delta_{p,p'}\delta_{q,q'} p!q!\innp{\tilde{f},\tilde{g}},\text{\quad (It\^{o}'s isometry)}
   \label{ito isometry}
\end{align}
where the Kronecker delta $\delta_{p,p'}$ is 1 when $p'$ is equal to $p$, and is 0 otherwise, and
$\innp{\cdot,\cdot}$ is the inner product on $\FH^{\otimes(p+ q)} $.  As a consequence,
\begin{align}
     E[\abs{I_{p,q}(f) }^2]&=p!q!\norm{\tilde{f}}^2\le p!q!\norm{f}^2,\text{\quad (It\^{o}'s isometry)}.\label{symmetric2}
\end{align}
The proof of Theorem~\ref{equilent cond} proceeds through several propositions and lemmas.
Firstly, we define the contraction of $(i,j)$ indices of two symmetric functions.
\begin{defn}\label{contra}
For two symmetric functions $f\in \FH^{\odot p_1}\otimes \FH^{\odot q_1},\,g\in \FH^{\odot p_2}\otimes \FH^{\odot q_2}$ and $i\le p_1\wedge q_2,\,j\le q_1\wedge p_2$, the contraction of $(i,j)$ indix  is defined as (see \cite{ch})
   \begin{align*}
    & f\otimes_{i,j} g (t_1,\dots,t_{p_1+p_2-i-j};s_1,\dots,s_{q_1+q_2-i-j})\\
    &=\int_{\Rnum_{+}^{2l}} \dif \vec{u}\dif \vec{u'} \phi(u_1,u_1')\dots \phi(u_i,u_i')f(t_1,\dots,t_{p_1-i},u_1,\dots,u_i; s_1\dots,s_{q_1-j},v_1\dots,v_j)\\
    & \times g(t_{p_1-i+1},\dots,t_{p-l},v_1',\dots,v_j'; s_{q_1-j+1}\dots,s_{q-l},u_1',\dots,u_i')\phi(v_1,v_1')\dots \phi(v_j,v_j')\dif \vec{v}\dif \vec{v'},
   \end{align*}where $l=i+j,\, p=p_1+p_2,\,q=q_1+q_2$, $\vec{u}=(u_1,\cdots, u_i),\,\vec{u'}=(u_1',\cdots, u_i')$ and $\vec{v}=(v_1,\cdots,v_j),\,\vec{v'}=(v_1',\cdots,v_j')$.
\end{defn}
By convention, $f\otimes_{0,0} g=f\otimes g$ denotes the tensor product of $f$ and $g$.
We write $f\tilde{\otimes}_{p,q} g$ for the symmetrization of $f\otimes_{p,q} g$. In what follows, we use the convention $f\otimes_{i,j} g=0$ if $i> p_1\wedge q_2$ or $j> q_1\wedge p_2$.

Our next result is a technical lemma.
\begin{lem}\label{lem 42}
Suppose that $F=I_{m,n}(f)$ with $f\in \FH^{\odot m}\otimes \FH^{\odot n}$ and that $\bar{F}=I_{n,m}(h)$. Then
\begin{align}
   & \E[\abs{F}^4]-2\big(\E[\abs{F}^2]\big)^2 -\abs{E[{F}^2]}^2\nonumber\\
      &=\sum_{0 <i+j<l}{m\choose i}^2{n\choose j}^2 (m!n!)^2  \norm{f\otimes_{i,j}f}^2_{\FH^{\otimes(2(l-i-j))}} +\sum_{r=1}^{l-1}((l-r)!)^2\norm{\psi_r}^2_{\FH^{\otimes(2(l-r))}}\label{f 4 q 1}\\
      &=\sum_{0<i+j<l}{m\choose i}{n\choose i}{n\choose j}{m\choose j} (m!n!)^2 \norm{f\otimes_{i,j}h}^2_{\FH^{\otimes(2(l-i+j))}}\label{f 4 q 2}\\
      &+\sum_{r= 1}^{l-1} (2m-r)!(2n-r)!  \norm{\varphi_r}^2_{\FH^{\otimes 2(l-r)}},\nonumber
\end{align}
where $l=m+n$ and
\begin{align}
   \psi_r&=\sum_{i+j=r} i!j! {m\choose i}^2{n\choose j}^2 f\tilde{\otimes}_{i,j} h,\label{bbb} \\
   \varphi_r&=\sum_{i+j=r} i!j! {m\choose i}{n\choose i}{n\choose j}{m\choose j} f\tilde{\otimes}_{i,j} f.
\end{align}
\end{lem}
\begin{proof}
It follows from Lemma 4.1 of \cite{ch} that
\begin{eqnarray}
     && \E[\abs{F}^4]-\big(\E[\abs{F}^2]\big)^2\label{aaa}\\
      &=&\sum_{i+j>0}{m\choose i}^2{n\choose j}^2 (m!n!)^2  \norm{f\otimes_{i,j}f}^2_{\FH^{\otimes(2(l-i-j))}} +\sum_{r=1}^{l}\big((l-r)!\big)^2\norm{\psi_r}^2_{\FH^{\otimes(2(l-r))}}.\nonumber
   \end{eqnarray}
We calculate the term $\psi_{l}=f\otimes_{m,n}h $:
\begin{eqnarray}
 f\otimes_{m,n}h
  &=&\int_{\Rnum_{+}^{m+n}}d\vec{u}  d\vec{u'} d\vec{v} d\vec{v'}\phi(u_1,u_1')\dots\phi(u_m,u_m')\phi(v_1,v_1')\dots\phi(v_n,v_n')\nonumber \\
  && \times f(u_1,\dots,u_m;v_1,\dots, v_n)h(v_1',\dots, v_n';u_1',\dots,u_m') \nonumber\\
  &=&\norm{f}^2_{\FH^{\otimes(m+n)}}  = \frac{1}{m!n!}\E[\abs{F}^2],\label{fk}
\end{eqnarray}
 where the last equality is from Ito's isometry (\ref{symmetric2}).
Next, we calculate the term $f\otimes_{m,n}f $ in Eq.(\ref{aaa}) according to whether $m=n$ or not.
We consider the case  $m\neq n$ first.  Without loss generality we can take $m>n$. By Definition~\ref{contra} we have that if $i> n$ or $j>n$ then $   f\otimes_{i,j}f=0.$
Therefore, if $m\neq n$ then
\begin{equation}\label{fk2}
   f \otimes_{m,n}f = 0 = \E[F_k^2],
\end{equation} where the last equality is   It\^{o}'s isometry (\ref{ito isometry}).
If $m= n$, similarly to show (\ref{fk}), we obtain that
\begin{align}\label{mequ n}
f\otimes_{m,m}f=\innp{f,\,h}_{\FH^{\otimes(m+n)}}=\frac{1}{(m!)^2}\E[F^2].
\end{align}
Substituting (\ref{mequ n}) or (\ref{fk2}) according to whether $m=n$ or not, and (\ref{fk}), into (\ref{aaa}), we obtain (\ref{f 4 q 1}). Similarly, we can show   (\ref{f 4 q 2}).
\end{proof}
\begin{nott}
   Suppose that $f(\vec{t}^{m},\vec{s}^n)\in \FH^{\odot m}\otimes \FH^{\odot n}$. Denote
   \begin{equation}
      f_u(\vec{t}^{m-1},\vec{s}^n)=f(\vec{t}^{m-1},u,\vec{s}^n),\qquad f^v(\vec{t}^{m},\vec{s}^{n-1})=f(\vec{t}^{m},\vec{s}^{n-1},v).
   \end{equation}
   Clearly, $f_u(\vec{t}^{m-1},\vec{s}^n)\in \FH^{\odot m-1}\otimes \FH^{\odot n}$ and $f^v(\vec{t}^{m},\vec{s}^{n-1})\in \FH^{\odot m}\otimes \FH^{\odot n-1}$.
\end{nott}

\begin{defn}[Complex Malliavin Derivatives]\label{dfn32}
Let $\mathcal{S}$ denote the set of all random variables of the form
\begin{equation}\label{definition}
  f\big(\zeta^H(\varphi_1),\cdots,\zeta^H(\varphi_m)\big),
\end{equation}
where $f\in C_{\uparrow}^\infty(\Cnum^m)$ and $\varphi_i\in\FH,i=1,2,\cdots,m$. Let $F\in \mathcal{S}$ be given by (\ref{definition}). The complex Malliavin derivative of $F$ is the element of $L^2(\Omega, \mathfrak{H} )$ defined by:
 \begin{align}
    D  F     & =\sum_{i=1}^m \partial_{i } f(\zeta^H(\varphi_1),\dots,\zeta^H(\varphi_m))\varphi_{i},\\
    \bar{D}  F     & =\sum_{i=1}^m \bar{\partial}_{i } f(\zeta^H(\varphi_1),\dots,\zeta^H(\varphi_m))\bar{\varphi}_{i},
  \end{align}
  where $\partial_j f= \frac{\partial}{\partial z_j}f(z_1,\dots,z_m) ,\quad \bar{\partial}_j f=\frac{\partial}{\partial \bar{z}_j}f(z_1,\dots,z_m)$ are the Wirtinger derivatives \cite{camp}.
\end{defn}
\begin{prop}\label{etaxinu}
Suppose that $l=m+n$ and
\begin{align}
 \eta_r&=\sum_{i+j=r}i{m\choose i}^2{n\choose j}^2i!j!f\tilde{\otimes}_{i,j}h,\label{etak}\\
 \xi_r& =\sum_{i+j=r}j {m\choose i}^2{n\choose j}^2i!j!h\tilde{\otimes}_{j, i}f,\label{xik}\\
 \nu_r&=\sum_{i+j=r} i {m\choose i}{n \choose i}{n\choose j}{m\choose j}i!j!f\tilde{\otimes}_{i, j}f,\label{nuk}
\end{align} then we have that
\begin{align}
   \mathrm{Var}(\norm{D I_{m,n}(f)}_{\FH}^2)&= \sum_{r=1}^{l-1} [(l-r)!]^2 \norm{\eta_r}^2_{\FH^{\otimes (2(l-r) )}},\label{var1}\\
   \mathrm{Var}(\norm{\bar{D} I_{m,n}(f)}_{\FH}^2)&= \sum_{r=1}^{l-1} [(l-r)!]^2 \norm{\xi_r}^2_{\FH^{\otimes (2(l-r) )}},\label{var2}\\
   \mathrm{Var}(\innp{D I_{m,n}(f),\, {D} \overline{I_{m,n}(f)}}_{\FH} )&=\sum_{r=1}^{l-1} (2m-r)! (2n-r)!  \norm{\nu_r}^2_{\FH^{\otimes (2(l-r) )}}\label{var3}.
\end{align}
\end{prop}
\begin{proof}We need only to show (\ref{var1}) since the other two are similar.
Denote $l'=m+n-1$.\\
{\it Step 1: Using product formula.}
By Theorem 12(D) of \cite{ito} and the product formula of complex Wiener-It\^{o} multiple integrals (Theorem 3.2 of \cite{ch}), 
we have that
\begin{align*}
   & \frac{1}{m^2}\norm{D_{\cdot}(I_{m,n}(f))}_{\FH}^2\\
   &=\norm{I_{m-1,n}\big(f_u(\vec{t}^{m-1},\vec{s}^n)\big)}^2_{\FH}\\
   &=\int_{[0,\infty)^2}\dif u\dif v \phi(u,v) I_{m-1,n}\big(f_u(\vec{t}^{m-1},\vec{s}^n)\big)\overline{I_{m-1,n}\big(f_v(\vec{t}^{m-1},\vec{s}^n)\big)}\\
   &=\sum_{i=0}^{m-1}\sum^{n}_{j=0}\,{m-1\choose i}^2{n\choose j}^2i!j!\,\int_{[0,\infty)^2}\dif u\dif v \phi(u,v) I_{l'-i-j,l'-i-j}\big(f_u(\vec{t}^{m-1},\vec{s}^n)\otimes_{i,j}h^v(\vec{t}^n,\vec{s}^{m-1})\big)
\end{align*}
where $h^v(\vec{t}^n,\vec{s}^{m-1})=\bar{f}_v(\vec{s}^{m-1},\vec{t}^n)=\bar{f}(\vec{s}^{m-1},v,\vec{t}^n)$ and
\begin{align}\label{fuhv}
   & \Big( f_u(\vec{t}^{m-1},\vec{s}^n)\otimes_{i,j}h^v(\vec{t}^n,\vec{s}^{m-1})\Big)(\bar{t}^{l'-i-j},u,\bar{s}^{l'-i-j},v)\\
   &=\int_{[0,\infty)^{i+j}}\dif \vec{x}^i\dif\vec{x'}^i \phi(x_1,x_1')\dots \phi(x_i,x_i')f_u(\vec{t}^{m-1-i},\vec{x}^i,\vec{s}^{n-j},\vec{y}^j )\nonumber \\
  &\times \bar{f}_v(s_{n-j+1},\dots,s_{l'-j-i},\vec{x'}^i, t_{m-i},\dots,t_{l'-i-j},\vec{y'}^j)  \phi(y_1,y_1')\dots \phi(y_j,y_j')\dif \vec{y}^j\dif\vec{y'}^j.\nonumber
\end{align}
Then we obtain that
\begin{align}\label{1m2}
   \frac{1}{m^2}\norm{D_{\cdot}(I_{m,n}(f))}_{\FH}^2=\sum_{r=0}^{l'}\int_{[0,\infty)^2}\dif u\dif v \phi(u,v) I_{l'-r,l'-r}(g_r(u,v)).
\end{align} where
\begin{equation}
   g_k(u,v)=\sum_{i+j=k} {m-1\choose i}^2{n\choose j}^2i!j!f_u(\vec{t}^{m-1},\vec{s}^n)\tilde{\otimes}_{i,j}h^v(\vec{t}^n,\vec{s}^{m-1}).\label{guv}
\end{equation}
Taking expectation to Eq.(\ref{1m2}), we have that
\begin{align}
  \frac{1}{m^2}\E[\norm{D_{\cdot}(I_{m,n}(f))}_{\FH}^2]&= \int_{[0,\infty)^2}\dif u\dif v \phi(u,v) g_{l'}(u,v)\nonumber\\
  &=(m-1)!n!\int_{[0,\infty)^2}\dif u\dif v \phi(u,v)f_u\otimes_{m-1,n}h^v\nonumber\\
  &=(m-1)!n! \norm{f}_{\FH^{\otimes (m+n)}}^{2}.\label{norm dd}
\end{align}
{\it Step 2: Calculating variance.} It follows from Fubini's theorem and It\^{o}'s isometry that we have:
\begin{align}\label{d equav}
   &\frac{1}{m^4}\mathbf{E}[\norm{D_{\cdot}\big( I_{m,n}(f)\big)}_{\FH}^4]\\&=\sum_{r=0}^{l'} \int_{[0,\infty)^4}\dif u\dif v\dif u'\dif v'\phi(u,v)\phi(u',v')[(l'-r)!]^2 \innp{{g_r(u,v)},g_r(u',v')}_{\FH^{\otimes 2(l'-r)}}.\nonumber
\end{align}
It is easy to check that
\begin{align*}
& \int_{[0,\infty)^4}\dif u\dif v\dif u'\dif v'\phi(u,v)\phi(u',v') \innp{f_u\tilde{\otimes}_{i,j} h^v,f_{u'}\tilde{\otimes}_{i,j} h^{v'}}_{\FH^{\otimes 2 (l'-k)}}\\
   &=\innp{f \tilde{\otimes}_{i+1,j} h ,f \tilde{\otimes}_{i+1,j} h }_{\FH^{\otimes 2 (l'-k)}} =\norm{f\tilde{\otimes}_{i+1,j} h}^2_{\FH^{\otimes 2(l'-k)}},
\end{align*}
which implies that
\begin{align*}
   \frac{1}{m^4}\mathbf{E}[\norm{D_{\cdot}\big( I_{m,n}(f)\big)}_{\FH}^4]=\sum_{r=0}^{l'}[(l'-r)!]^2  \innp{G_r,\,G_r}_{\FH^{\otimes 2(l'-r)} },
\end{align*}where
\begin{align}
    G_k =\sum_{i+j=k} {m-1\choose i}^2{n\choose j}^2i!j!f(\vec{t}^{m},\vec{s}^n)\tilde{\otimes}_{i+1,j}h(\vec{t}^n,\vec{s}^{m}).\label{Gg}
\end{align}
Especially, for the term with $k=l'$, we have that 
\begin{align*}
    \abs{G_{l'}}^2&=[(m-1)!n!]^2\abs{f\otimes_{m,n} h}^2\\
    &=[(m-1)!n!]^2\norm{f}_{\FH^{\otimes (m+n)}}^4
    =\Big(\frac{1}{m^2}\mathbf{E}[\norm{D_{\cdot}(I_{m,n}(f))}_{\FH}^2]\Big)^2.
\end{align*}
Substituting the above equality displayed into (\ref{d equav}), we have that
\begin{align*}
   \mathrm{Var}(\norm{D I_{m,n}(f)}_{\FH}^2)&=m^4\sum_{r'=0}^{l'-1}[(l'-r')!]^2  \innp{G_{r'},\,G_{r'}}_{\FH^{\otimes 2(l'-r')} }\\
   &=   \sum_{r=1}^{l-1} [(l-r)!]^2 \norm{\eta_r}^2_{\FH^{\otimes (2(l-r) )}} \qquad (\text{ let\,\, } l=l'+1,\,r=r'+1),
\end{align*} where
$   \eta_r=m^2 G_{r'}$, which implies the desired expressions (\ref{etak})-(\ref{var1}).

\end{proof}


\noindent{\it Proof of Theorem~\ref{equilent cond}.\,} Since (i)$\Rightarrow $(ii) is well known, we need only to show the following implications:
\begin{align*}
   (\mathrm{ii})\Rightarrow (\mathrm{iii})\Rightarrow (\mathrm{iv})\Rightarrow (\mathrm{i}).
\end{align*}

[(ii)$\Rightarrow$ (iii) ] Condition (ii) implies that as $k\to \infty$,
\begin{align}\label{f4 moment}
  \E[\abs{F_k}^4]-2\big(\E[\abs{F_k}^2]\big)^2 -\abs{E[{F_k}^2]}^2\to 0,
\end{align} which implies that Condition (iii) holds by (\ref{f 4 q 1})-(\ref{f 4 q 2}), (see Lemma~\ref{lem 42}).

[(iii)$\Rightarrow $(iv)] Suppose that Condition (iii) holds. The inequality (5.2) of \cite{ito} implies that as $k\to \infty$,
\begin{align*}
   \norm{f_k\tilde{\otimes}_{i,j} f_k}_{\FH^{\otimes ( 2(l-i-j))}}\to 0,\qquad
   \norm{f_k\tilde{\otimes}_{i,j} h_k}_{\FH^{\otimes ( 2(l-i-j))}}=\norm{h_k\tilde{\otimes}_{i,j} f_k}_{\FH^{\otimes ( 2(l-i-j))}}\to 0.
\end{align*}
Therefore, it follows from Minkowski's inequality and Proposition~\ref{etaxinu} that as $k\to \infty$,
\begin{align*}
 \eta_r^{k} \to 0 ,\qquad \xi_r^{k}\to 0,\qquad \nu_r^{k}\to 0,\qquad r=1,\dots,l-1,
\end{align*}where $\eta_r^{k},\,\xi_r^{k},\,\nu_r^{k} $ are given as Equations (\ref{etak})-(\ref{nuk}).
By (\ref{var1})-(\ref{var3}), we obtain that Condition (iv) holds.

[(iv)$\Rightarrow$(i)] We follow the idea of \cite[Theorem 4]{nourorti}, i.e. we combine Malliavin calculus and partial differential equations. Let
\begin{align*}
   \varphi_k(z)=\E\big[e^{\mi(\bar{z}F_k+z\bar{F}_k)/2} \big].
\end{align*} Then we have that
\begin{equation}\label{par varphik}
    \left\{
      \begin{array}{l}
    \frac{\partial \varphi_k}{\partial z}=\frac{\mi}{2}\E\big[\bar{F}_k\times e^{\mi(\bar{z}F_k+z\bar{F}_k)/2} \big],    \\
    \frac{\partial \varphi_k}{\partial \bar{z}}=\frac{\mi}{2}\E\big[ F_k\times e^{\mi(\bar{z}F_k+z\bar{F}_k)/2} \big].
      \end{array}
\right.
\end{equation}

By the assumption $\E[\abs{F_k}^2]\to \sigma^2$, $\set{F_k}$ are tight. Now suppose that the subsequence $\set{F_{n_k}}$ converges to $G$ in law. Without ambiguity, we still denote $\set{F_{n_k}}$ by $\set{F_{k}}$. By the hypercontractivity inequality of complex multiple Wiener-It\^{o} integrals \cite{ch},  $\set{\abs{F_k}^r}$ is uniformly integrable and thus $\E[\abs{G}^r]=\lim_{k\to \infty}\E[\abs{F_k}^r]$ for all $r\ge 1$ \cite[Theorem 5.4]{bll}. Therefore, the characteristic function $\varphi(z)=\E[e^{\frac{\mi}{2}(\bar{z}G+z\bar{G})}]$ has continuous partial derivatives of any order.

It is not difficult to see that
\begin{align*}
   \E\big[\bar{F}_k\times e^{\mi(\bar{z}F_k+z\bar{F}_k)/2} \big]&=\frac{1}{m}\E\big[\overline{(\delta D) {F}_k}\times e^{\mi(\bar{z}F_k+z\bar{F}_k)/2} \big]\\
   &=\frac{1}{m}\E\big[\innp{ D (e^{\mi(\bar{z}F_k+z\bar{F}_k)/2}),\, D {F}_k}_{\FH} \big]\\
   &=\frac{1}{m}\E\big[\innp{e^{\mi(\bar{z}F_k+z\bar{F}_k)/2}(\bar{z}DF_k+z D\bar{F}_k),\, D{F}_k}_{\FH} \big]
\end{align*}
Clearly, for any $z\in \Cnum$, $e^{\mi(\bar{z}F_k+z\bar{F}_k)/2}\to e^{\mi(\bar{z}G+z\bar{G})/2} $ in $L^2(\Omega)$. Thus, Condition (iv) implies that as $k\to \infty$,
\begin{align*}
   &\E\big[\innp{e^{\mi(\bar{z}F_k+z\bar{F}_k)/2}(\bar{z}DF_k+zD\bar{F}_k),\, D{F}_k}_{\FH} \big]\\
   &\to (\bar{z}\lim_{k\to\infty}\E[\norm{D F_k}^2_{\FH}]+z\lim_{k\to\infty}\E[\innp{D\bar{F}_k\,D{F}_k}_{\FH} ])\varphi(z),\quad \forall z\in \Cnum,
\end{align*}
since the scalar product in the Hilbert space $L^2(\Omega)$ depends continuously on its factors.
It follows from (\ref{norm dd}) and
\begin{align*}
     \frac{1}{mn} \E\big[\innp{D I_{m,n}(f),\, {D} \overline{I_{m,n}(f)}}_{\FH}\big]=\delta_{m,n}m!(m-1)!f\otimes_{m,m}f,
\end{align*}
that
\begin{align*}
   \lim_{k\to\infty}\E[\norm{D F_k}^2_{\FH}]&=m\lim_{k\to\infty}\E[\norm{F_k}^2_{\FH}]=m\sigma^2,\\
    \lim_{k\to\infty}\E[\innp{D\bar{F}_k\,D{F}_k}_{\FH}]&=m\delta_{m,n}\lim_{k\to\infty}\overline{\E[F_k^2]}\\
    &=m(c-\mi b)\delta_{m,n} .
\end{align*}
Therefore, it follows from (\ref{par varphik}) that for any $z\in \Cnum$,
\begin{align}
   \frac{\partial \varphi}{\partial z}= [\bar{z} \sigma^2+z \cdot(c-\mi b)\delta_{m,n} ]\varphi(z).
\end{align}
In the same way,
\begin{align}
   \frac{\partial \varphi}{\partial\bar{ z}}= [\bar{z}\cdot(c+\mi b)\delta_{m,n}+z \sigma^2 ]\varphi(z).
\end{align} Clearly, $\varphi(0)=1$. Therefore, $G$ is a bivariate normal distribution with covariance matrix $\tensor{C}=\frac{1}{2}\begin{bmatrix} \sigma^2+c & b\\ b & \sigma^2-c  \end{bmatrix}$. Prokhorov's theorem implies that
$\set{F_k}$ converges to a bivariate normal distribution with the desired covariance matrix $\tensor{C}$.
{\hfill\large{$\Box$}}\\
\section{Asymptotic  consistency and normality }
We need several   propositions and lemmas before the
  proof of Theorem~\ref{strong law}.
The following lemma's proof is easy.
\begin{lem}\label{deri}
   For any $H\in (\frac12,1)$, we have that
\begin{equation}
  \int_{[0,\infty)^2}e^{-\gamma u_1-\bar{\gamma} u_2}\abs{u_1-u_2}^{2H-2}\dif u_1\dif u_2=d\,,
\end{equation}
where $d$ is defined by \eqref{d d}.
\end{lem}
\begin{prop}\label{prop 5 1}
Let $Z$ be the solution to \eqref{cp}.   As $T\to \infty$, we have that
   \begin{equation}
     \frac{1}{T}\int_0^T \abs{Z_t}^2\dif t\to  {a}\alpha_H{d},\qquad a.s.
   \end{equation}
\end{prop}
\begin{proof}
Denote $Y_t=\sqrt{a}\int_{-\infty}^t e^{-\gamma (t-s)}\dif \zeta_s$. It is easy to see that $Y$ is
centered complex Gaussian process. It\^{o}'s isometry implies that for any $t\in \Rnum, s\ge 0$,
  \begin{align}
      E[Y_{t+s}\bar{Y}_t]={\alpha_H}a\int_{-s}^{ \infty}\dif v_1 \int_0^{ \infty} \dif v_2 e^{-\gamma (v_1+s)}e^{-\bar{\gamma} v_2 } \abs{v_1-v_2}^{2H-2} =E[Y_{s}\bar{Y}_0].\label{star}
   \end{align}
Thus $Y_t$ is stationary. It is easy to check that as $s\to \infty$, $E[Y_{s}\bar{Y}_0]\to 0$ with the same order as $\abs{s}^{2H-2}$, which implies that $\set{Y_t}$ is ergodic \cite[p78]{dym}.

Then we have that
\begin{align*}
   Z_t &=e^{-\gamma t}Z_0+\sqrt{a}\int_0^t e^{-\gamma (t-s)}\dif \zeta_s\\
   &=e^{-\gamma t}Z_0+\sqrt{a}\int_{-\infty}^t e^{-\gamma (t-s)}\dif \zeta_s-e^{-\gamma t}\sqrt{a}\int_{-\infty}^0 e^{\gamma s}\dif \zeta_s\\
   &=Y_t+e^{-\gamma t}(Z_0- Y_0)\,.
\end{align*}
The ergodicity and Cauchy-Schwarz inequality imply that as $T\to \infty$,
\begin{align*}
   \frac{1}{T}\int_0^T\abs{Z_t}^2\dif t&=\frac{1}{T}\int_0^T\Big[\abs{Y_t}^2-2\Re (e^{-\gamma t} (Z_0-Y_0) \bar Y_t)+e^{-2\lambda t}\abs{ Z_0-Y_0 }^2\Big]\dif t\\
   &\to\lim_{T\rightarrow \infty} \E[\abs{Y_T}^2] =  {a}{\alpha_H}d, \qquad a.s. 
\end{align*}
where the last equality is from Eq.(\ref{star}) and Lemma~\ref{deri}.
\end{proof}

Denote
\begin{equation}
\psi_t(r,s)=e^{-\bar{\gamma}(r-s)}\mathbf{1}_{\{0\le s\le r\le t\}}
\quad {\rm and}\quad    X_t=I_{1,1}(\psi_t(r,s)) \,.
\label{e.def-X}
\end{equation}

\begin{lem}\label{sub sequence}
As $n\to \infty$, the sequence $\set{\xi_n:=\frac{1}{ n} X_{n}}$ converges to zero almost sure.
\end{lem}
\begin{proof}Denote $F_T=\frac{1}{\sqrt{T}}X_T$.
Lemma~\ref{lem 5_2 ht} implies that $
  \sup_{n }\E[\abs{F_{n} }^2]<\infty$.
From the    hypercontractivity of multiple Wiener-It\^{o} integrals \cite[Proposition 2.4]{ch}, we see that
$
  \sup_{n }\E[\abs{F_{n} }^4]<\infty$.
For any fixed $\varepsilon>0$, 
it follows from Chebyshev's inequality that
\begin{align*}
   P\big(\abs{\xi_n}> \varepsilon \big) =P\big(\abs{F_{n}}>\sqrt{n}\varepsilon \big) \le  \frac{1}{ {n}^2\varepsilon^4 }   \E[\abs{F_{n} }^4] \le \frac{3^4}{ {n}^2\varepsilon^4 }   \E[\abs{F_{n} }^2]^2\,.
   \end{align*}
The Borel-Cantelli lemma implies that $\set{\xi_n}$ converges to zero almost surely.
\end{proof}

\begin{prop}\label{holder continu}
For any real number $p\ge 2$ and integer $n\ge 1$,
\begin{align}
  B_n&:=\int_{n}^{n+1}\int_{n}^{n+1}\frac{\abs{X_t-X_s}^p}{\abs{t-s}^{2pH}} \dif s\dif t \label{bnbn}
\end{align}
is finite. Moreover, for any real numbers
$p>2$, $q>1$ and integer $n\ge 1$,
\begin{align}
   \abs{X_{t_2}-X_{t_1}}&\le R_{p,q} n^{q/p} ,\quad \forall t_1, t_2\in[n, n+1],
\end{align}
where $R_{p,q}$ is a random constant  independent of $n$.
\end{prop}
\begin{proof}
   For any $n\le t_1\le t_2\le n+1$, It\^o's isometry implies that
\begin{align*}
 &  \E[\abs{X_{t_2}-X_{t_1}}^2] =\norm{\psi_{t_2}(r,s)-\psi_{t_1}(r,s)}^2_{\FH^{\otimes 2}} =\norm{e^{-\bar{\gamma}(r-s)}\mathbf{1}_{\{t_1< s\le r\le t_2 \}}}^2_{\FH^{\otimes 2}}\\
   &=\int_{t_1< s_1\le r_1\le t_2}\int_{t_1< s_2\le r_2\le t_2}e^{-\bar{\gamma}(r_1-s_1)-\gamma(r_2-s_2)} \phi(r_1,r_2)\phi(s_1,s_2)\dif s_1\dif s_2\dif r_1\dif r_2\\
   &\le \int_{t_1< s_1\le r_1\le t_2}\int_{t_1< s_2\le r_2\le t_2}\phi(r_1,r_2)\phi(s_1,s_2)\dif s_1\dif s_2\dif r_1\dif r_2
=(t_2-t_1)^{4H}.
\end{align*}
The hypercontractivity of multiple Wiener-It\^{o} integrals \cite[Proposition 2.4]{ch} implies that for any $p\ge 2$ and any  $n\le t_1\le t_2\le n+1$,
\begin{align*}
    \E[\abs{X_{t_2}-X_{t_1}}^p]&\le(p-1)^p\E[\abs{X_{t_2}-X_{t_1}}^2]^{\frac{p}{2}}\\
    &\le   (p-1)^p (t_2-t_1)^{2p H} .
\end{align*}
%
Take $\Psi(x)=x^p$ and $\rho(x)=x^{2H}$.  The above inequality yields
\begin{align}\label{pp1}
\E (B_n)=   \E\Big[\int_{n}^{n+1}\int_{n}^{n+1}\Psi\big(\frac{\abs{X_t-X_s}}{\rho(\abs{t-s})}\big)\dif s\dif t\Big]\le (p-1)^p.
\end{align}
For any $q>1$, we have
\[
\E   \left( \sum_{n=1}^\infty \frac{ B_n}{n^q}\right)
=\sum_{n=1}^\infty \frac{ \E(B_n)}{n^q}<\infty\,.
\]
This implies that
\[
\sum_{n=1}^\infty \frac{ B_n}{n^q}\le R_{p,q} \quad \hbox{for some random constant $R_{p,q}$}\,.
\]
Or we have
\begin{equation}
 B_n \le R_{p, q} n^q\quad \hbox{for all positive number
 $q>1$ and integer $n\ge 1$}\,.  \label{e.bn-bound}
\end{equation}
An application of the Garsia-Rodemich-Rumsey inequality \cite[Theorem 2.1]{hyz 17}  implies that
\begin{align*}
   \abs{X_t-X_s}&\le 8\int_0^{\abs{t-s}}\Psi^{-1}(\frac{4B_n}{u^2})\rho'(u)\dif u =16H\frac{(4B_n)^{1/p}}{2H-\frac{2}{p}}\abs{t-s}^{2H-\frac{2}{p}}
   \le c_{p } B_n^{\frac{1}{p}}\,.
\end{align*}
This combined with \eqref{e.bn-bound} proves  the proposition.
\end{proof}

Denote
   \begin{align}\label{ft ht}
       h_t(r,s)=e^{- {\gamma}(-r+s)}\mathbf{1}_{\{0\le r\le s\le t\}}\,.
   \end{align}
\begin{lemma}\label{lem 5_2 ht}
Let $H\in (\frac12,\frac34)$. Then  the following integrals are absolutely convergent
  \begin{align}
     &\lim_{T\to \infty}\frac{1}{\alpha_H^2 T} \int_{[0,T]^4}\psi_T(t_1,s_1)\overline{\psi_T(t_2,s_2)}\phi(t_1,t_2)\phi(s_1,s_2)\dif t_1\dif t_2 \dif s_1\dif s_2 =\sigma^2  \nonumber
\\
     &\lim_{T\to \infty}\frac{1}{\alpha_H^2 T}  \int_{[0,T]^4}\psi_T(t_1,s_1)\overline{h_T(t_2,s_2)}\phi(t_1,t_2)\phi(s_1,s_2)\dif t_1\dif t_2 \dif s_1\dif s_2= c +\mi b\,, \nonumber
  \end{align}
  where $\sigma^2$ and $c, b$ are defined by \eqref{limlim1} and
  \eqref{ab c}.
\end{lemma}
\begin{proof}
   We only evaluate the first integral since the other one is similar.
We divide the domain $\set{0\le s_1\le t_1\le T,0\le s_2\le t_2\le T}$  into six disjoint regions according to the distinct orders of $s_1,t_1,s_2,t_2$:
\begin{eqnarray*}
\Delta_1&=&\left\{0\le s_2\le t_2\le s_1\le t_1\le T\right\}\,,\qquad
\Delta_2=\left\{
0\le  s_1\le t_1\le s_2\le t_2 \le T\right\}\\
\Delta_3&=& \left\{0\le  s_1\le s_2\le t_1\le t_2 \le T\right\}\,,  \qquad
\Delta_4=\left\{ 0\le s_2\le  s_1\le t_2\le t_1 \le T\right\}\,,  \\
\Delta_5&=&\left\{ 0\le  s_1\le s_2\le t_2\le t_1 \le T\right\} \,,
\qquad \Delta_6=\left\{  0\le s_2\le  s_1\le t_1\le t_2 \le T \right\}\,.
\end{eqnarray*}
We also denote $I_i=\frac{1}{T}\int_{\Delta_i} \psi_T(t_1,s_1)\overline{\psi_T(t_2,s_2)}\phi(t_1,t_2)\phi(s_1,s_2)\dif t_1\dif t_2 \dif s_1\dif s_2$,
$i=1, \cdots, 6$.

   Firstly, we consider $I_1$.
   It follows from L'Hospital rule that
   \begin{align}
      &\lim_{T\to \infty} I_1=\lim_{T\to \infty} \int_0^T \dif s_1 \int_{0}^{s_1} \dif t_2\int_{0}^{t_2} \dif s_2 e^{-\bar{\gamma}(T-s_1)}e^{-\gamma(t_2-s_2)}\phi(T,t_2)\phi(s_1,s_2)\,.
   \end{align}
Making substitution  $a=t_2-s_2,\,b=s_1-t_2,\,c=T-s_1$,   we have that as $T\to \infty$
   \begin{align}
     & \lim_{T\rightarrow \infty}I_1=\alpha_H^2 \lim_{\rightarrow \infty} \int_{a,b,c\ge 0,a+b+c\le T} \dif a\dif b \dif c e^{-\bar{\gamma} c} e^{-\gamma a}[(b+c)(a+b)]^{2H-2}\nonumber\\
     &= \alpha_H^2 \int_{[0,\infty)^3} \dif a\dif b \dif c\, e^{-\bar{\gamma} c} e^{-\gamma a}[(b+c)(a+b)]^{2H-2}.\label{limitt 54}
     \end{align}
The above integral is absolutely convergent when $H\in (\frac12,\frac34)$. In fact, since
   \begin{equation*}
    (b+c)(a+b)\ge ac\mathbf{1}_{[0,1]}(b) +b^2\mathbf{1}_{[1,\infty)}(b),
   \end{equation*}we have that
   \begin{align*}
     & \abs{\int_{[0,\infty)^3} \dif a\dif b \dif c\, e^{-\bar{\gamma} c} e^{-\gamma a}[(b+c)(a+b)]^{2H-2}}\\
&\le  \int_{[0,\infty)^3} \dif a\dif b \dif c\,e^{-\lambda c} e^{-\lambda a} \big[ b^{4H-4}\mathbf{1}_{[1,\infty)}+ (ac)^{2H-2}\mathbf{1}_{[0,1]}(b)\big]\\
&=\frac{1}{(3-4H)\lambda^2}+\big(\frac{\Gamma(2H-1)}{\lambda^{2H-1}}\big)^2.
   \end{align*}
  Substituting the equality of Gamma function $\int_0^{\infty} e^{-x\beta}x^{\alpha -1}=\frac{\Gamma(\alpha)}{\beta^\alpha}$ with $\alpha>0,\,\Re{\beta}>0$ into (\ref{limitt 54}), we have that
\begin{align}
  & \int_{[0,\infty)^3} \dif a\dif b \dif c\, e^{-\bar{\gamma} c} e^{-\gamma a}[(b+c)(a+b)]^{2H-2}\nonumber\\
 &=\frac{1}{\Gamma(2-2H)^2}\int_{[0,\infty)^5} \dif a\dif b \dif c\dif x\dif y\, e^{-\bar{\gamma} c} e^{-\gamma a}e^{-x (b+c)}x^{1-2H}e^{-y (a+b)}y^{1-2H}\nonumber\\
  &=\frac{1}{\Gamma(2-2H)^2}\int_{[0,\infty)^2} \dif x\dif y \frac{(xy)^{1-2H}}{(x+y)(x+\bar{\gamma})(y+\gamma)}.\label{limt 33}
\end{align}

It is easy to see that  $I_2=I_1$.   In a similar  way as for $I_1$, we have
\begin{eqnarray}
\lim_{T \to \infty}   I_3
&=&\lim_{T \to \infty}   I_4
 =\frac{\alpha_H^2\Gamma^2(2H-1)}{2\lambda \abs{\gamma}^{4H-2}},\\
\lim_{T \to \infty}   I_5
&=& \frac{\alpha_H^2\Gamma^2(2H-1)}{2\lambda \bar{\gamma}^{4H-2}},\\
 \lim_{T \to \infty}   I_6
 &=&\frac{\alpha_H^2\Gamma^2(2H-1)}{2\lambda {\gamma}^{4H-2}}.\label{limitt 77}
\end{eqnarray}

Finally, by adding (\ref{limt 33})-(\ref{limitt 77}) together, we get (\ref{limlim1}).
\end{proof}

\begin{lemma}\label{con equa 0}
Let $\psi_T,\,h_T$ be as in (\ref{e.def-X}) and (\ref{ft ht}) respectively. As $T\to \infty$, we have that:
\begin{equation}
   \frac{1}{T}\psi_T\otimes_{0,1}\psi_T\to 0,\, \frac{1}{T}\psi_T\otimes_{1,0}\psi_T\to 0,\, \frac{1}{T}\psi_T{\otimes}_{0,1}h_T\to 0,\, \frac{1}{T}\psi_T{\otimes}_{1,0}h_T\to 0,\, \text{in }\, \FH^{\otimes 2}.
\end{equation}
\end{lemma}
\begin{proof}
When $0<t<s<T$,  we have that
\begin{eqnarray}
&&\frac{1}{\alpha_H}  \left|\psi_T\otimes_{0,1}\psi_T(t,s)
   \right|
 =\left|  \int_0^t\dif u_1 \int_s^T\,\dif u_2 e^{-\bar{\gamma}(t-u_1)}e^{-\bar{\gamma}(u_2-s)}\abs{u_1-u_2}^{2H-2}\right|
 \nonumber  \\
&\le&   \int_0^t\dif u_1 \int_s^T\,\dif u_2 e^{-\lambda(t-u_1)}e^{-\lambda (u_2-s)}\abs{u_2-u_1}^{2H-2}\nonumber \\
 &\le& (s-t)^{2H-2}\int_0^t\dif u_1 \int_s^T\,\dif u_2 e^{-\lambda(t-u_1)}e^{-\lambda (u_2-s)}
  \le\frac{1}{\lambda^2} (s-t)^{2H-2}\,. \label{s ge t}
\end{eqnarray}
%

When $s< t$, we have that
\begin{eqnarray*}
    & &\int_0^t\dif u_1 \int_s^T\,\dif u_2 e^{-\lambda(t-u_1)}e^{-\lambda (u_2-s)}\abs{u_2-u_1}^{2H-2}\\
  &=&\Big(\int_0^s\dif u_1 \int_s^T\,\dif u_2+\int_s^t\dif u_1 \int_s^t\,\dif u_2+\int_s^t\dif u_1 \int_t^T\,\dif u_2\Big) e^{-\lambda(t-u_1)}\\
  &&\qquad e^{-\lambda (u_2-s)}\abs{u_2-u_1}^{2H-2}=I_1(T)+I_2(T)+I_3(T)\,.
\end{eqnarray*}
For the first term, we have that
\begin{align*}
  I_1(T)
   &=e^{-\lambda (t-s)}\int_0^s\dif u_1 \int_{s-u_1}^{T-u_1}\,\dif z  e^{-\lambda z}z^{2H-2}\\
   &=e^{-\lambda (t-s)}\int_{0}^{T}\,\dif z  e^{-\lambda z}z^{2H-2}\int_{0\vee (s -z)}^{s\wedge (T-z)}\dif u_1 \\
   &\le e^{-\lambda (t-s)}\int_{0}^{T}\,\dif z  e^{-\lambda z}z^{2H-2}[s- (s -z)] \\
   &\le e^{-\lambda (t-s)}\frac{\Gamma(2H)}{\lambda^{2H}}\le c_{\lambda,H} \abs{t-s}^{2H-2},
\end{align*} where $c_{\lambda,H}$ is a constant independent of $T$.
For the second term, we have that
\begin{align*}
 I_2(T)
   &= 2 e^{-\lambda (t-s)} \int_s^t\dif u_1 \int_s^{u_1}\,\dif u_2 e^{\lambda(u_1-u_2)}(u_1-u_2)^{2H-2}\\
   &= 2 e^{-\lambda (t-s)} \int_s^t\dif u_1 \int_{0}^{u_1-s}\,\dif z e^{\lambda z}z^{2H-2}\\
   &=2 e^{-\lambda (t-s)} \int_0^{t-s}\dif z e^{\lambda z}z^{2H-2}(t-s-z) \le c_{\lambda,H} \abs{t-s}^{2H-2},
\end{align*} where $c_{\lambda,H}$ is a constant independent of $T$ and the last inequality is by means of L'Hospital rule. In fact, when $H\in (\frac12,1)$, then we have that $$\lim_{x\to\infty}\frac{ \int_0^{x}\dif z e^{\lambda z}z^{2H-2}(x-z)}{e^{\lambda x} x^{2H-2}}=\frac{1}{\lambda^2}.$$
For the third term, we have that
\begin{align*}
  I_3(T)
   &=  e^{-\lambda (t-s)} \int_s^t\dif u_1 \int_{t-u_1}^{T-u_1}\,\dif z e^{-\lambda z}z^{2H-2}\\ 
   &= e^{-\lambda (t-s)} \int_0^{T}\dif z e^{-\lambda z}z^{2H-2} \int_{s\vee (t-z)}^{t\wedge(T-z)}\,\dif u_1\\
   &\le e^{-\lambda (t-s)} \int_0^{t-s}\dif z e^{-\lambda z}z^{2H-2}[t-(t-z)]\\
   &\le e^{-\lambda (t-s)}\frac{\Gamma(2H)}{\lambda^{2H}}\le c_{\lambda,H} \abs{t-s}^{2H-2},
\end{align*} where $c_{\lambda,H}$ is a constant independent of $T$. Thus,  we have that
\begin{align}
   &\abs{\int_0^t\dif u_1 \int_s^T\,\dif u_2 e^{-\bar{\gamma}(t-u_1)}e^{-\bar{\gamma}(u_2-s)}\abs{u_1-u_2}^{2H-2} }
  \le c_{\lambda,H} \abs{t-s}^{2H-2}.\label{kerc}
\end{align}
This  inequality   together with the inequality (\ref{s ge t}) implies that
\begin{align*}
   \norm{\frac{1}{T}\psi_T\otimes_{0,1}\psi_T}^2_{\FH^{\otimes 2}}&\le \frac{\alpha_H^2 c_{\lambda,H}^2}{T^2}\norm{\phi}^2_{\FH^{\otimes 2}}\\
   &=\frac{\alpha_H^2 c_{\lambda,H}^2}{T^2}\int_{[0,T]^4}\dif t_1\dif t_2\dif s_1\dif s_2 \phi(t_1,s_1){\phi}(t_2,s_2)\phi(t_1,t_2)\phi(s_1,s_2).
\end{align*}
As $T\to \infty$, L'Hospital rule and the symmetric property of the above integrand imply that when $H\in (\frac12,\frac34)$,
\begin{align*}
   \lim_{T\to \infty} \norm{\frac{1}{T}\psi_T\otimes_{0,1}\psi_T}^2_{\FH^{\otimes 2}}&\le \lim_{T\to \infty}\frac{2\alpha_H^2 c_{\lambda,H}^2}{T}\int_{[0,T]^3}\dif t_2\dif s_1\dif s_2 \phi(T,s_1)\phi(t_2,s_2)\phi(T,t_2)\phi(s_1,s_2)\\
   &=\lim_{T\to \infty} \frac{ 2\alpha_H^2 c_{\lambda,H}^2}{T^{6-8H}}\int_{[0,1]^3}\dif t_2\dif s_1\dif s_2 \phi(1,s_1)\phi(t_2,s_2)\phi(1,t_2)\phi(s_1,s_2)\\
   &=0 .
\end{align*}

Finally, it is easy to obtain that $\psi_T\otimes_{1, 0}\psi_T=f_T\otimes_{0,1}f_T$. Thus $\frac{1}{T}\psi_T\otimes_{1, 0}\psi_T\to 0$ also holds as $T\to \infty$. In addition, it follows from Lemma 5.4 of web-only Appendix of \cite{huNua} that both $\frac{1}{T}\psi_T\otimes_{1, 0}h_T\to 0$ and $\frac{1}{T}\psi_T\otimes_{0,1}h_T\to 0 $ hold.
\end{proof}

\noindent{\it Proof of Theorem~\ref{strong law}.\,}
Without loss of generality, we can suppose that $Z_0=0$. By (\ref{hat gamma}), we obtain that
\begin{equation}\label{hat gamm-gamm}
    \hat{\gamma}_T-\gamma=\sqrt{a}\frac{\frac{1}{ {T}}X_T}{\frac{1}{T}\int_0^T \abs{Z_t}^2\dif t}.
\end{equation}
By Proposition~\ref{prop 5 1}, we need only to show $\frac{1}{ {T}} X_T$ converges to zero almost sure as $T\to \infty$. Clearly, we have that
\begin{align}
   \abs{\frac{1}{{T}}X_T}&\le \frac{1}{T}\abs{X_T-X_n }+ {\frac{n}{T}} \frac{1}{ {n}} \abs{X_n}.\label{2 epsilon}
\end{align}
where  $n=[T]$ is the biggest integer less than or equal to a real number $T$.
Using Lemma~\ref{sub sequence} and since $n/T$ is bounded, we see
  that  the second term in
  (\ref{2 epsilon}) goes to $0$ almost surely  as $T\to\infty$.

By Proposition~\ref{holder continu}, we see that
the first term in  (\ref{2 epsilon}) is
bounded by $\frac{1}{T} R_{p,q} n^{q/p}$ for any $p\ge 2 $ and
$q>1$. Choosing  $q<p$   we see that the  first
term in (\ref{2 epsilon}) goes to $0$ as $T\rightarrow\infty$.
This completes the proof of the first part of Theorem
\ref{strong law}.

Now we turn to the proof of the second part. Denote $F_T=\frac{1}{\sqrt{T}}X_T$. Clearly,
    \begin{align}
       \bar{F}_T=\frac{1}{\sqrt{T}}\int_{[0,T]^2} e^{- {\gamma}(-r+s)}\mathbf{1}_{\{r\le s\}} \dif  \zeta_r \dif\bar{\zeta}_s=\frac{1}{\sqrt{T}}I_{1,1}(h_T(r,s)).
    \end{align}
From Theorem~\ref{equilent cond}, Lemma~\ref{lem 5_2 ht} and Lemma~\ref{con equa 0}, we see \[
\hbox{ $F_T$ converges in law to $\zeta\sim \mathcal{N}(0,\frac{
   \sigma_H^2}{2}\tensor{C})$,}
   \]
    where $\tensor{C}$ as in Theorem~\ref{strong law}.
We write  Equation (\ref{hat gamm-gamm}) as
\begin{equation*}
   \sqrt{T}(\hat{\gamma}_T-\gamma)=\sqrt{a}\frac{F_T}{\frac{1}{T}\int_0^T \abs{Z_t}^2\dif t}.
\end{equation*}
Therefore, it follows from the above fact,
 Proposition~\ref{prop 5 1}, and Slutsky's theorem that
$ \sqrt{T}(\hat{\gamma}_T-\gamma)$  converges in distribution to bivariate Gaussian law $\mathcal{N}(0,\frac{1}{2d^2a}\tensor{C})$.
{\hfill\large{$\Box$}}\\
\vskip 0.2cm {\small {\bf  Acknowledgements}:
Y. Chen is supported by the China Scholarship Council (201608430079);
Y. Hu is partially supported by the Simons Foundation (209206);
Z. Wang is supported by the Mathematical Tianyuan Foundation of China (11526117) and the Zhejiang Provincial Natural Science Foundation (LQ16A010006).}


\end{document}